\documentclass[reqno]{amsart}

\usepackage[T1]{fontenc}
\usepackage{graphicx}
\usepackage{mathptmx}      
\usepackage{pinlabel}
\usepackage{amsmath}
\usepackage{amsfonts}
\usepackage{amssymb}
\usepackage{enumerate}
\usepackage{expdlist}
\usepackage{tikz}
\usetikzlibrary{decorations.pathreplacing}
\usetikzlibrary{calc}
\usetikzlibrary{shapes.geometric}
\usepackage{empheq}
\usepackage{tkz-euclide}
\usepackage{pgfplots}
\pgfplotsset{compat=1.14}
\usepgfplotslibrary{ternary}

\usepackage{amsthm}
\usepackage[nooneline,hang]{subfigure}
\usepackage[margin=0pt,font=normalsize,singlelinecheck=off, 
labelfont=bf]{caption}
\usepackage{graphicx}
\usepackage{bbold}
\usepackage{dsfont}
\usepackage{floatflt}
\usepackage{mathrsfs}
\usepackage{url}

\allowdisplaybreaks[1]

\newcommand{\cal}{\mathcal}
\newcommand{\R}{\mathbb R}
\newcommand{\C}{\mathbb C}
\newcommand{\N}{\mathbb N}
\newcommand{\Z}{\mathbb Z}

\newcommand{\eps}{\varepsilon}

\newcommand{\cro}{\mathrm{cr}}

\theoremstyle{plain}
\newtheorem{theorem}{Theorem}[section]

\newtheorem{corollary}[theorem]{Corollary}
\newtheorem{lemma}[theorem]{Lemma}
\theoremstyle{definition}
\newtheorem{remark}[theorem]{Remark}
\theoremstyle{definition}
\newtheorem{definition}[theorem]{Definition}

\title[$C^\infty$-convergence of conformal mappings on triangular 
lattices]{$C^\infty$-convergence of conformal mappings for conformally 
equivalent triangular lattices}

\author{Ulrike B\"ucking}

\date{\today}
\begin{document}

\begin{abstract}
Two triangle meshes are conformally equivalent if for any pair of incident 
triangles the absolute values of the corresponding cross-ratios of the four 
vertices agree. 
Such a pair can be considered as preimage and image of a discrete conformal 
map. In this article we study discrete conformal maps which are defined 
on parts of a triangular lattice $T$ with strictly acute angles. That is, $T$ 
is an infinite triangulation of the plane with congruent strictly acute 
triangles. A smooth conformal map $f$ can be approximated on a compact subset by 
such discrete conformal maps $f^\eps$, defined on a part of $\eps T$, 
see~\cite{Bue16}. We improve this 
result and show that the convergence is in fact in~$C^\infty$. Furthermore, we 
describe how the cross-ratios of the four vertices for pairs of incident 
triangles are related to the Schwarzian derivative of~$f$.
\end{abstract}

\maketitle

\section{Introduction}
Holomorphic functions build the basis and heart of the rich theory of complex 
analysis. The subclass of {\em conformal maps} consists of holomorphic 
functions 
with nowhere vanishing derivatives. These may be characterized as infinitesimal 
scale-rotations. M\"obius transformations are special conformal maps on the 
Riemann sphere $\hat{\C}$, which preserve all cross-ratios. Recall that the 
cross-ratio of four distinct points $a,b,c,d\in\C$ is defined as
\begin{equation*}
 \cro(a,b,c,d)=\frac{(a-b)(c-d)}{(b-c)(d-a)}.
\end{equation*}
A conformal map $f$ infinitesimally preserves cross-ratios. Additionally, the 
first deviation from being a M\"obius transformation can be expressed by the 
Schwarzian derivative of $f$, which is defined as
\begin{equation}\label{eqdefSchwarzian}
{\mathfrak S}[f](z)= \left(\frac{f''(z)}{f'(z)}\right)' 
-\frac{1}{2}\left(\frac{f''(z)}{f'(z)}\right)^2
= \frac{f'''(z)}{f'(z)}-\frac{3}{2}\left(\frac{f''(z)}{f'(z)}\right)^2.
\end{equation}
In particular, there holds
 \begin{multline*}
 \lim_{\eps\to 0}\frac{1}{\eps^2}\left( 
\frac{\cro(f(a),f(a+\eps(b-a)),f(a+\eps(c-a)),f(a+\eps(d-a)))}{ \cro(a,b,c,d)} 
-1\right) \\
=\frac{(a-c)(b-d)}{6}{\mathfrak S}[f](a).
\end{multline*}

\subsection{$C^\infty$-convergence for discrete conformal maps on triangular 
lattices}
In the discrete theory, the idea of characterizing conformal maps as local
scale-rotations may be translated into different concepts. Here we consider the 
discretization coming from a metric viewpoint: Infinitesimally, lengths are 
scaled by a factor, i.e.\ by $|f'(z)|$ for a conformal function $f$ on 
$D\subset\C$. The smooth complex domain is replaced in this discrete setting by 
a triangulation of a connected subset of the plane $\C$. The infinitesimal 
preservation of the cross-ratios is then substituted by the preservation of all 
length cross-ratios ($=$ absolute values of the cross-ratio) for all pairs of 
incident triangles. (Note that only M\"obius transformations would preserve all 
cross-ratios of pairs of incident triangles of the triangulation. So this 
condition would be two restrictive.)

In this article we consider the case where the triangulation is a (part of a) 
{\em triangular lattice}. In particular, let $T$ be a lattice triangulation of 
the whole complex plane $\C$ with congruent triangles, see 
Figure~\ref{FigTiling}. 
\begin{figure}[t]
\subfigure[Example of a triangular lattice.]{\label{FigTiling}
\includegraphics[width=0.5\textwidth]{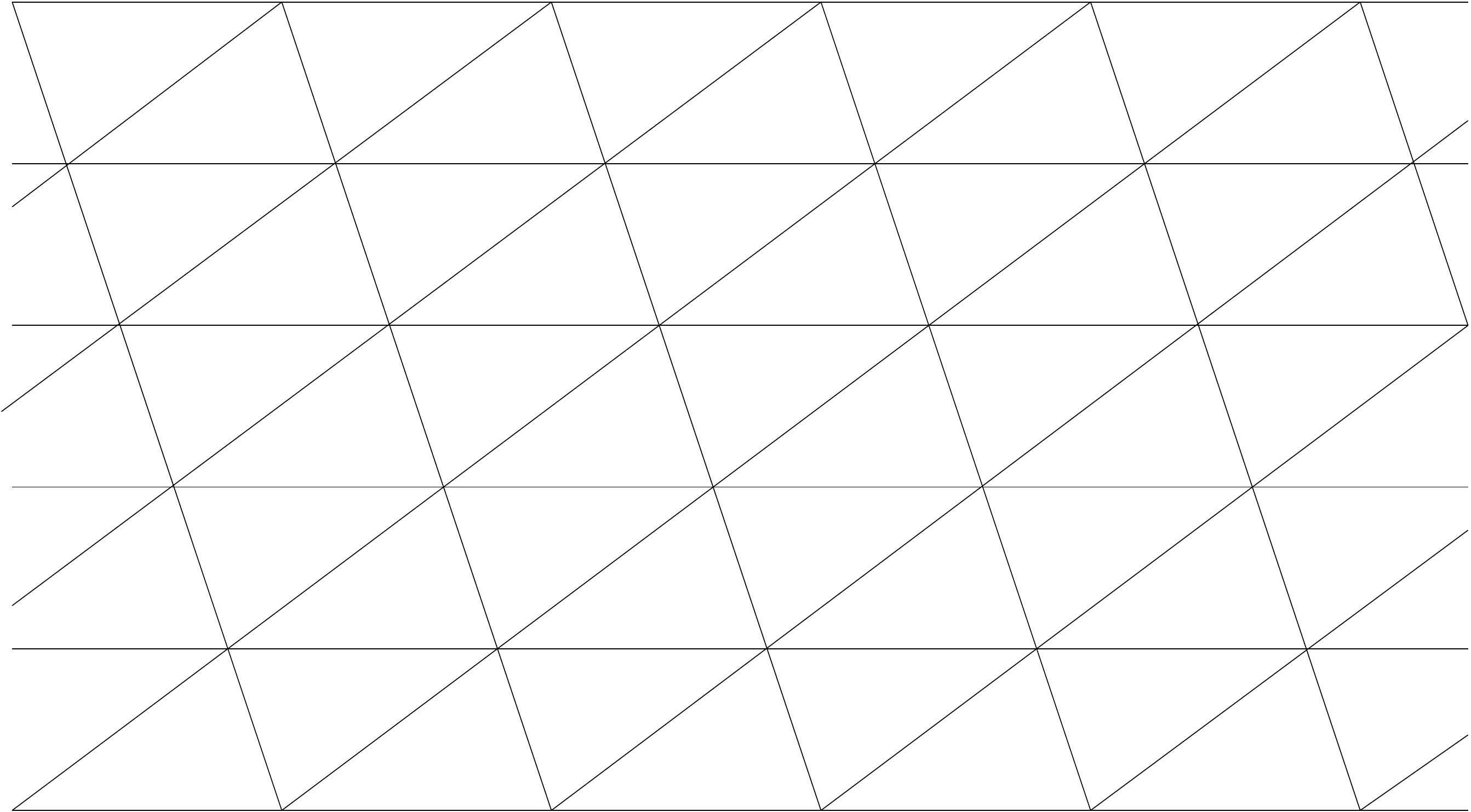}
}
\hspace{3em}
\subfigure[Suitably scaled acute angled triangle.]{\label{figTriang}
 \input{TriangleABCS.pspdftex}
}
\caption{Lattice triangulation of the plane with congruent 
triangles.}\label{FigRegTile}
\end{figure}
The sets of vertices and edges of $T$ are denoted by $V$ and $E$ respectively.
Edges will often be written as $e=[v_i,v_j]\in E$, where $v_i,v_j\in V$ are 
its incident vertices. For triangular faces we use the notation 
$\Delta[v_i,v_j,v_k]$ enumerating the incident vertices with respect to the
orientation (counterclockwise) of $\C$. We only consider the case of 
acute angles, i.e., $\alpha,\beta,\gamma\in(0,\pi/2)$ and 
assume for simplicity that the origin is a vertex.

On a subcomplex of $T$ we now define a discrete conformal mapping by the 
preservation of the length cross-ratios.

\begin{definition}[see~\cite{BPS13}]\label{defdcm}
 A {\em discrete conformal map} $g$ is the restriction to the vertices $V_S$ of 
a continuous and orientation preserving map $g_{\text{PL}}$ of a subcomplex 
$T_S$ of a triangular lattice $T$ to $\C$. We demand that $g_{\text{PL}}$ is 
locally a homeomorphism in a 
neighborhood of each interior point and that its restriction to every triangle
is a linear map onto the corresponding image triangle, that is the mapping is 
piecewise linear. 
Furthermore, the absolute value 
of the cross-ratio (called {\em length cross-ratio}) is preserved for all 
pairs of adjacent triangles:
\begin{equation}
 |\cro(v_1,v_2,v_3,v_4)|= |\cro(g(v_1),g(v_2),g(v_3),g(v_4))|,
\end{equation}
where $\Delta[v_1,v_2,v_3]$ and $\Delta[v_1,v_3,v_4]$ are two adjacent 
triangles of the lattice with common edge $[v_1,v_3]$ and $|a|$ denotes the 
modulus of $a\in\C$.
\end{definition}

Note that the values of the cross-ratio $\cro(g(v_1),g(v_2),g(v_3),g(v_4))$ on 
all interior edges $[v_1,v_3]$ determine the map $g$ up to a global M\"obius 
transformation, see also Remark~\ref{remqtof} below for more details.

\begin{remark}[\cite{BPS13}]
 For a continuous, orientation preserving and piecewise linear map 
$g_{\text{PL}}$ on a simply 
connected subcomplex the preservation of the length cross-ratios is equivalent 
to the existence of a function $u:V_S\to\R$ on the vertices, 
called {\em associated scale factors}, such that for all edges $e=[v,w]\in E_S$ 
there holds
\begin{equation}\label{eqdefdiscf}
 |g(v)-g(w)|=|v-w|\text{e}^{(u(v)+u(w))/2}.
\end{equation}
Thus the lengths of the edges 
of the triangulation are changed according to scale factors at the vertices. 
The new triangles are then ``glued together'' to result in a piecewise linear 
map, see Figure~\ref{FigExMap} for an illustration.
\end{remark}

\begin{figure}[ht]
\centering{
\includegraphics[width=0.3\textwidth, trim={0 2cm 0 2.7cm},clip]{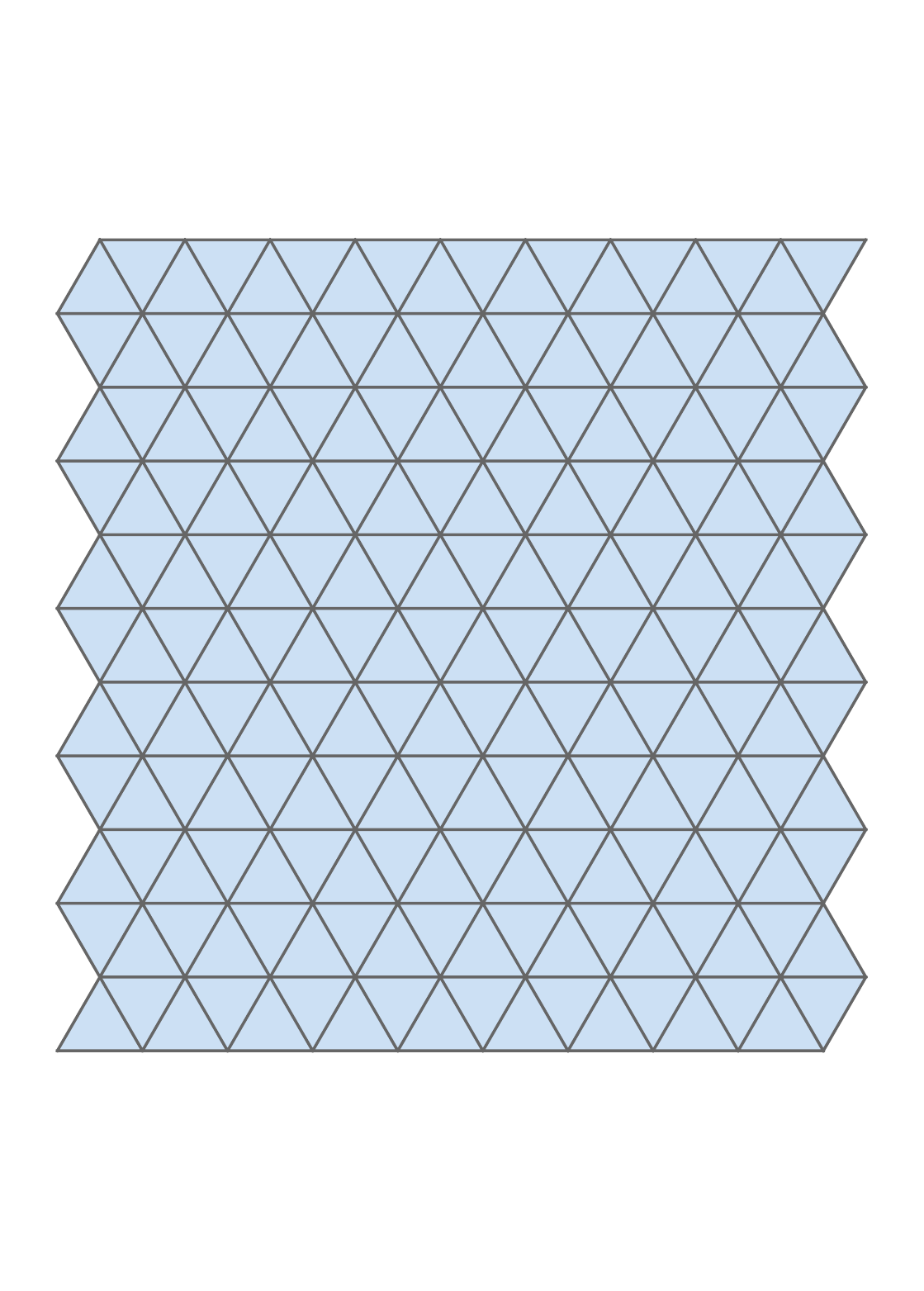} 
\begin{minipage}[b][0.18\textwidth][t]{0.2\textwidth}
 \[\overset{g}{\longrightarrow} \]
\end{minipage}
\includegraphics[width=0.3\textwidth, trim={0 1.7cm 0 1.7cm},clip]{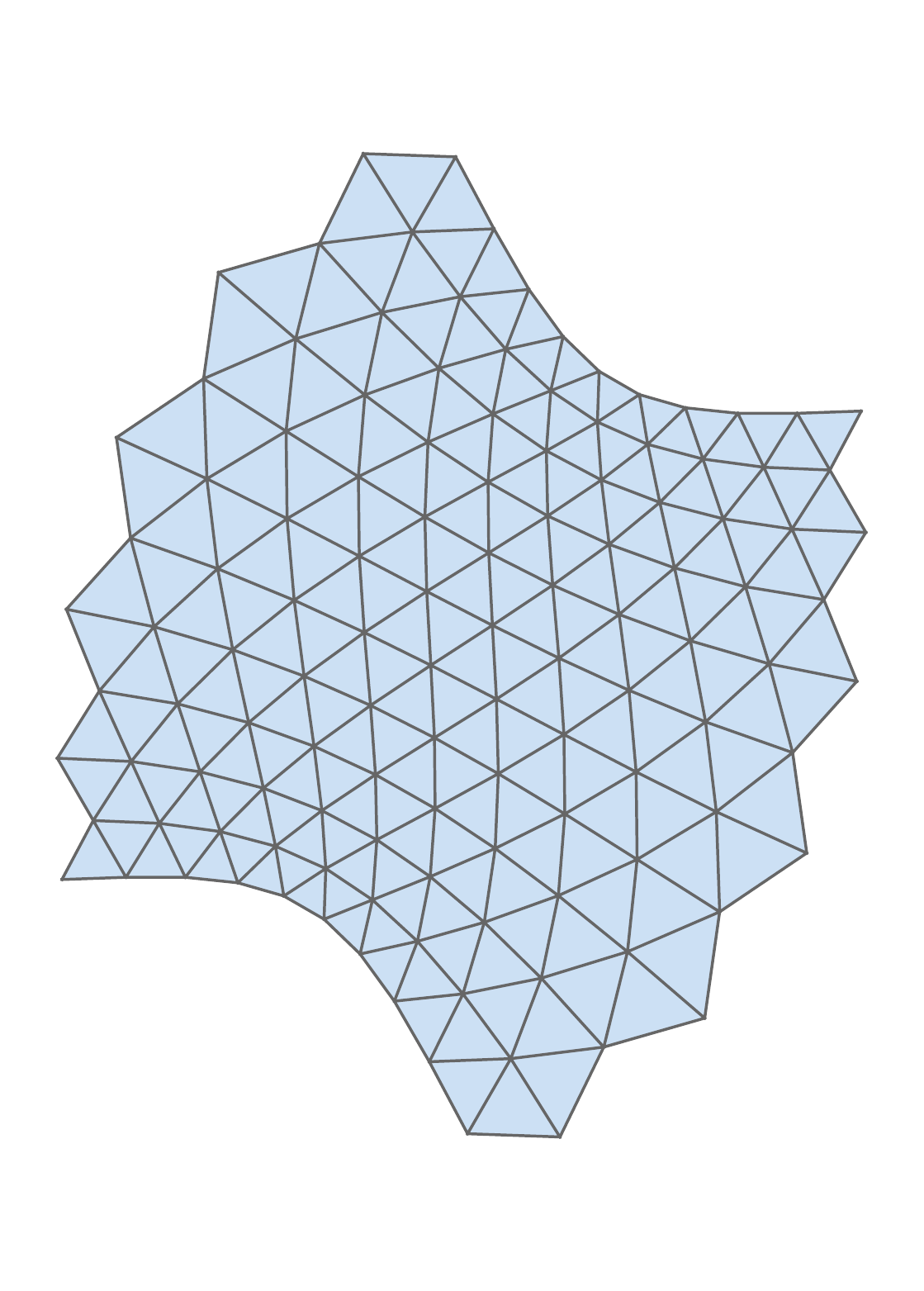}
}
\caption{Example of a discrete conformal map $g$.}\label{FigExMap}
\end{figure}

In fact, our definition of a discrete conformal map relies on the notion of 
discretely conformally equivalent triangle meshes. These have been studied by  
Luo, Gu, Sun, Wu, Guo~\cite{Luo04,Luo13,Luo14}, Bobenko, Pinkall, 
and Springborn~\cite{BPS13} and others.


In~\cite{Bue16} we showed that given a smooth conformal map $f$ there exists a 
sequence of discrete conformal maps $f^\eps$ which approximates the given map 
on a compact set.
In particular, the discrete conformal maps can be obtained from a Dirichlet 
problem: Given some function $u_\partial$ on the boundary of a subcomplex $T_S$,
find a discrete conformal map whose associated
scale factors agree on the boundary with $u_\partial$. Of course, we choose the 
boundary values  $u_\partial$ according to the given function $f$ as 
$u_\partial=\log|f'|$.
In this article we improve this result and show that the approximation in 
fact is $C^\infty$.
Furthermore, as a by-product, we establish the approximation of the 
Schwarzian derivative of $f$ using cross-ratios of pairs of incident triangles.

To be precise, denote by $\eps T$ the lattice $T$ scaled by $\eps>0$.

\begin{theorem}[{\cite[Theorem~1.1]{Bue16}}]\label{theoConvalt}
Let $f:D\to\C$ be a conformal map (i.e.\ holomorphic with $f'\not=0$). Let 
$K\subset D$ be a compact set which is the closure of its simply connected 
interior $\Omega_K:=int(K)$ and assume that $0\in \Omega_K$. 
Let $T$ be a triangular lattice with strictly acute angles.
For each $\eps>0$ let $T^\eps_K$ 
be a subcomplex of $\eps T$ whose support is contained in $K$ and is 
homeomorphic to a closed disc. We further assume that $0$ is an 
interior vertex of $T^\eps_K$. Let $e_0=[0,{\mathbb v_{\mathbb 0}}]\in 
E^\eps_K$ be one of its incident edges.

Then if $\eps>0$ is small enough (depending on 
$K$, $f$, and $T$) there exists a unique discrete conformal map $f^\eps$ on 
$T^\eps_K$ which satisfies the following two conditions:
\begin{itemize}
 \item The associated scale factors $u^\eps:V^\eps_K\to\R$ satisfy
\begin{equation}\label{eqboundu}
 u^\eps(v)=\log|f'(v)|\qquad \text{for all boundary vertices } v \text{ of } 
V^\eps_K.
\end{equation}
\item The discrete conformal map is normalized according to 
\[\qquad f^\eps(0)=f(0)\quad \text{and}\quad \arg(f^\eps({\mathbb v_{\mathbb 
0}})-f^\eps(0))= \arg({\mathbb v_{\mathbb 
0}})+ \arg(f'(\frac{{\mathbb v_{\mathbb 0}}}{2})) 
\pmod{2\pi}.\]
\end{itemize}
Furthermore, the following estimates for $u^\eps$ and $f^\eps$ hold for all 
vertices $v\in V^\eps_K$ and points $x$ in the support of $T^\eps_K$ 
respectively with constants 
$C_1,C_2$ depending only on $K$, $f$, and $T$, but not on $v$ or $x$:
\begin{enumerate}[(i)]
 \item The scale factors $u^\eps$ approximate $\log |f'|$ uniformly with error 
of order~$\eps^2$:
\begin{equation}\label{eqconvu}
 \left|u^\eps(v)-\log|f'(v)|\right|\leq C_1\eps^2.
\end{equation}
\item The discrete conformal maps $f^\eps$ 
converge to $f$ for $\eps \to 0$ uniformly with error of order~$\eps$:
\begin{equation*}
 \left|f^\eps_{\text{PL}}(x)-f(x)\right|\leq C_2\eps,
\end{equation*}
where $f^\eps_{\text{PL}}$ is the piecewise linear extension of $f^\eps$ from 
Definition~\ref{defdcm}.
\end{enumerate}
\end{theorem}

In this article the subcomplexes $T^\eps_K$ will be chosen such that they 
approximate the compact set $K$. In particular, we will take for $T^\eps_K$ 
a subcomplex which is simply connected, contained in $K$ and contains $0$ and 
is ``as large as possible''. This means in particular, that adding any other 
triangle of 
$\eps T$ which is contained in $K$ and shares an edge with a triangle of 
$T^\eps_K$ will result in a subcomplex which ceases to be simply connected.

\begin{theorem}\label{theoConv}
Under the assumptions of Theorem~\ref{theoConvalt} and with the above 
definition of $T^\eps_K$, the discrete conformal maps 
converges in $C^\infty(\Omega)$ to $f$.
%
%
\end{theorem}


The proof of Theorem~\ref{theoConv} is inspired by the methods of the proof of 
$C^\infty$-convergence for hexagonal circle packings in~\cite{HeSch98}. In 
particular, the main objects are discrete Schwarzians defined in 
Section~\ref{secSchwarzian} as suitably scaled measure of deformation of the 
M\"obius invariant cross-ratios from their original values in the lattice $T$. 
As the discrete Laplacian of such a discrete Schwarzian is a polynomial in 
$\eps$ and the discrete Schwarzians, we can deduce their
$C^\infty$-convergence in Section~\ref{SecBound} analogously as 
in~\cite{HeSch98} from a Regularity lemma~\ref{lemReg} using some facts on 
discrete differential operators introduced in Section~\ref{SecPrelim}. The 
necessary boundedness of the discrete Schwarzian itself can be deduced from 
Theorem~\ref{theoConvalt}. Finally, the $C^\infty$-convergence of $f^\eps$ is 
shown in Section~\ref{SecConv}. Here, we also derive the precise connection 
between the limits of the discrete Schwarzians and the Schwarzian derivative of 
the given function $f$. In Section~\ref{SecGen} we discuss some generalizations 
of our proof, for example to the convergence of circle patterns with hexagonal 
combinatorics and other notions of discrete conformality.

\subsection{Other convergence results for discrete conformal maps}

Smooth conformal maps can be characterized in various ways. This leads to 
different notions of discrete conformality. Convergence issues have already
been studied for some of these discrete analogs. We only give a very short 
overview and cite some results of a growing literature.

Linear definitions can be derived 
as discrete versions of the Cauchy-Riemann equations and have a long and 
still developing history. Connections of such discrete mappings to smooth 
conformal functions have been studied for example 
in~\cite{CFL28,LF55,Mer07,ChS12,S13,BS15,W14}. In particular, this includes 
$C^\infty$-convergence for the regular $\eps\Z^2$-lattice.

The idea of characterizing conformal maps as local scale-rotations has lead to
the consideration of circle packings, more precisely to investigations on 
circle packings with the same (given) combinatorics of the 
tangency graph. Thurston~\cite{Thu85} first conjectured the convergence of 
circle packings to the Riemann map, which was then proven 
by~\cite{RS87,HeSch96,Ste97}. $C^\infty$-convergence for hexagonal circle 
packings was shown in~\cite{HeSch98}.

The theory of circle patterns generalizes the case of circle packings. Also, 
there is a link to integrable structures via isoradial circle 
patterns. The approximation of conformal maps using circle pattens has been 
studied in~\cite{Sch97} for orthogonal circle patterns with square grid 
combinatorics and furthermore in~\cite{DM05,diss,Bue08,LD07,BBS17}, which also 
contain results on $C^\infty$-convergence.

\section{Preliminaries on discrete differential operators and 
notation}\label{SecPrelim}

In the following, we introduce useful definitions and notation by generalizing
the notions defined in~\cite[Sec.~2]{HeSch98}.

We consider the regular triangular lattice $\eps T$ with edge length 
$\varepsilon>0$. In particular, let
\[V^\eps=\{n\eps\sin\alpha +m\text{e}^{i\beta}\eps\sin\gamma : n,m\in\Z\}\] 
be the set of vertices. We abbreviate the edge directions by 
\[\omega_1=1=-\omega_4,\qquad \omega_2=\text{e}^{i\beta}=-\omega_5,\qquad  
\omega_3=\text{e}^{i(\alpha+\beta)}=-\omega_6,\] 
and the corresponding edge lengths in $T$ as in Figure~\ref{figTriang} by 
\[L_1=\sin\alpha=L_4,\qquad L_2=\sin\gamma=L_5,\qquad L_3=\sin\beta=L_6.\]
Note in particular, that $L_1\omega_1-L_2\omega_2+L_3\omega_3=0$.

For $k=1,\dots,6$ denote by $\tau^\eps_k:V^\eps\to V^\eps$, 
$\tau^\eps_kv=v+\eps L_k\omega_k$ the translation along one of the lattice 
directions. For any subset $W\subseteq V^\eps$ a vertex $v\in W$ is called 
{\em interior vertex} of $W$ if all neighboring vertices $\tau^\eps_kv$ for 
$k=1,\dots, 6$ are contained in $W$. Set $W^{(0)}=W$ and for each $l\geq 1$ 
denote by $W^{(l)}$ the set of interior vertices of $W^{(l-1)}$.

Given a function $\eta:W\to\C$, denote by $\tau_k^\eps\eta$ the function which 
differs from $\eta$ by a translation $\tau_k^\eps$:
\[\tau_k^\eps\eta(v)=\eta(\tau_k^\eps v)= \eta(v+\eps L_k\omega_k).\] 
Define the {\em (discrete) directional derivative} $\partial_k^\eps\eta: 
W^{(1)}\to\R$ by
\[\partial_k^\eps\eta(v)=\frac{1}{\eps 
L_k}(\eta(v+\eps L_k\omega_k)-\eta(v)),\]
so $\partial_k^\eps=(\eps L_k)^{-1}(\tau^\eps_k-I)$, where $I\eta=\eta$.
For further use, note the following rule for the discrete differentiation of a 
product:
\begin{equation}\label{prodrule}
 \partial_k^\eps(\eta_1\eta_2) =(\partial_k^\eps \eta_1) \eta_2 + 
\tau_k^\eps \eta_1 (\partial_k^\eps \eta_2).
\end{equation}
Furthermore, define the {\em (discrete) Laplacian} $\Delta^\eps\eta: 
W^{(1)}\to\C$ by
\begin{equation}
 \Delta^\eps\eta(v)=\frac{1}{\eps^2}\sum\limits_{k=1}^6 
i\frac{\omega_k^2+1}{\omega_k^2-1} (\eta(v+\eps L_k\omega_k)-\eta(v)).
\end{equation}
Note that this is a scaled version of the well-known $\cot$-Laplacian as 
$\cot\varphi=i (\text{e}^{2i\varphi}+1)/(\text{e}^{2i\varphi}-1)$. Of 
course, the operators $I$, $\tau_k^\eps$, $\partial^\eps_j$ and $\Delta^\eps$ 
commute with each other.
We will also use $\|\eta\|_W=\sup_{v\in W}|\eta(v)|$ to denote the 
$L^\infty(W)$-norm of $\eta$.

Let $\Omega\subset\C$ be some domain and let $f:\Omega\to\C$ be some 
function. 
 For each $\eps>0$, let $f^\eps: 
W^\eps\to\C$ be some function defined on a set of vertices $W^\eps\subset 
V^\eps$. Assume that for every $z\in \Omega$ there are some 
$\delta_1,\delta_2>0$ 
such that for all $\eps\in(0,\delta_1)$ we have $\{v\in V^\eps : 
|v-z|<\delta_2\}\subset W^\eps$. 

Then we say that {\em $f^\eps$ converges to $f$ locally uniformly in $\Omega$}, 
if for every $\sigma>0$ and every $z\in\Omega$ there are $\delta_1,\delta_2>0$ 
such that $|f(z)-f^\eps(v)|<\sigma$ for every 
$\eps\in(0,\delta_1)$ and every $v\in W^\eps$ with $|v-z|<\delta_2$.

If $f$ is differentiable, denote by $\partial_k f$ the directional derivative, 
that is 
\[\partial_k f(z)=\lim\limits_{t\to 0} 
\frac{f(z+t\omega_k)-f(z)}{t}\qquad \text{ for } k=1,\dots,6.\]
Let $n\in\N$ and suppose that $f$ is $C^n$-smooth. We call {\em $f^\eps$ 
convergent to $f$ in $C^n(\Omega)$} if for every sequence 
$k_1,\dots,k_j\in\{1,\dots,6\}$ with $j\leq n$ the functions 
$\partial^\eps_{k_j}\partial^\eps_{k_{j-1}}\dots \partial^\eps_{k_1}f^\eps$ 
converges to $\partial_{k_j}\partial_{k_{j-1}}\dots \partial_{k_1}f$ 
locally uniformly in $\Omega$. If this holds for all $n\in\N$, the convergence 
is $C^\infty(\Omega)$.

The functions $f^\eps$ are called {\em uniformly bounded in $C^n(\Omega)$}, if 
for every compact set ${\cal K}\subset \Omega$ there is some constant 
$C({\cal K},n)$ such that $\|\partial^\eps_{k_j}\partial^\eps_{k_{j-1}}\dots 
\partial^\eps_{k_1}f^\eps\|_{W^\eps} < C(K,n)$ holds for every $j\leq n$ 
and all $\eps$ small enough. The functions $f^\eps$ are {\em uniformly 
bounded in $C^\infty(\Omega)$} if they are uniformly bounded in $C^n(\Omega)$ 
for all $n\in\N$.

The proofs of the following lemmas are simple adaptions 
of the corresponding arguments in~\cite[Sect.~2]{HeSch98}.

\begin{lemma}[see {\cite[Lemma~2.1]{HeSch98}}]\label{lemconv}
 Let $n\in\N$. Suppose that the functions $f^\eps$ are uniformly bounded in 
$C^{n+1}(\Omega)$. Then for every sequence $\eps\to 0$ there is a 
$C^n(\Omega)$-function $f$ and a subsequence of $\eps\to 0$ such that 
$f^\eps\to f$ in $C^n(\Omega)$ along this subsequence.
\end{lemma}

\begin{lemma}[see {\cite[Lemma~2.2]{HeSch98}}]
 Suppose that $f^\eps, g^\eps, h^\eps$ converges in $C^\infty(\Omega)$ to 
functions $f,g,h:\Omega\to\C$, defined on a domain $\Omega\supset K$, and 
suppose that 
$h\not=0$ in $\Omega$. Then the following convergences are in 
$C^\infty(\Omega)$:
\begin{enumerate}[(1)]
 \item $f^\eps+ g^\eps\to f+g$,
\item $f^\eps g^\eps\to fg$,
\item $1/h^\eps\to 1/h$,
\item if $h^\eps>0$ then $\sqrt{h^\eps}\to \sqrt{h}$,
\item $|h^\eps|\to |h|$.
\end{enumerate}
\end{lemma}

\section{The discrete Schwarzians}\label{secSchwarzian}
Let $f$ be a conformal map on a domain in the complex plane $\C$, that is, $f$ 
is a holomorphic function with non-vanishing derivative $f'(z)\not=0$. The 
Schwarzian derivative of $f$ is defined in~\eqref{eqdefSchwarzian} and is
itself holomorphic. Further, for any M\"obius 
transformation $M(z)=(az+b)/(cz+d)$ we have ${\mathfrak S}[M\circ f]= 
{\mathfrak S}[f]$ and 
${\mathfrak S}[f]=0$ if and only if $f$ is the restriction of some M\"obius 
transformation. For proofs and further properties of the Schwarzian see for 
example~\cite[Chap.~II]{Le87}.

In the following, we will define M\"obius invariants of conformally equivalent 
triangular lattices and derive their equations. Suitable M\"obius invariants 
and corresponding
equations have been worked out in~\cite{Sch97} for orthogonal circle patterns 
and in~\cite{HeSch98} for hexagonal circle packings.

Inspired by~\cite{HeSch98}, we will use the M\"obius invariants as intermediate 
means in the study of the convergence problem. The discrete Schwarzians will be 
defined as suitably scaled measure of deformation of the M\"obius invariants 
from their regular values. The convergence of the discrete Schwarzians is 
also notable on its own right and increases the connection between analogous 
notions for smooth and discrete conformal maps.

For any interior edge $[u,v]$ in $T^\eps_K$ with two adjacent triangles 
$\Delta[u,v,w_1]$ and $\Delta[u,w_2,v]$ denote by
\begin{align}
Q([u,v])& =\cro(u,w_2,v,w_1),\\
q^\eps([u,v])&= \cro(f^\eps(u),f^\eps(w_2),f^\eps(v) ,f^\eps(w_1)) 
\label{eqdefq}
\end{align}
the cross-ratio of the four vertices on the quad formed by the 
the two triangles $\Delta[u,v,w_1]$ and $\Delta[u,w_2,v]$ and by their images 
under $f^\eps$ respectively. Note that $Q([v,v+\eps L_k \omega_k])= 
(\omega_{k-1}L_{k-1})^2/(\omega_{k+1}L_{k+1})^2$ where the indices are taken 
modulo $6$.


We define the {\em discrete Schwarzian} at $[u,v]$ by
\begin{equation}
 s([u,v])=\frac{1}{\eps^2}\left(\frac{q^\eps([u,v])}{Q([u,v])}-1\right).
\end{equation}
If $f^\eps$ is a M\"obius transformation, we have $s([u,v])=0$, analogously to 
the smooth case.

For any vertex $v\in V^\eps$ denote $e_k(v)=[v,\tau_k^\eps(v)]$, see 
Figure~\ref{figFlowerMoeb}~(left). Let $q_k,s_k:(W^\eps)^{(1)}\to\C$ be 
defined as $q_k(v):=q^\eps(e_k(v))$ and $s_k(v):=s(e_k(v))$. Then obviously, 
$q_{k+3}(\tau_k^\eps v)=q_k(v)$ and 
\begin{equation}\label{eqsk+}
s_{k+3}(\tau_k^\eps v)=s_k(v),
\end{equation}
where the indices are taken modulo $6$. Note that $Q_k:= Q(e_k(v)) =Q_{k+3}$, 
as $\eps T$ is a lattice.

\begin{lemma}\label{lemeqq}
 Let $v$ be an interior vertex in $V_K^\eps$. Then there holds
\begin{align}
 & q_1(v)q_2(v)q_3(v)q_4(v)q_5(v)q_6(v)=1 \label{eqq1}\\
 &1-q_k(v) +q_k(v)q_{k+1}(v) -q_k(v)q_{k+1}(v)q_{k+2}(v) \notag \\
&\qquad +q_k(v)q_{k+1}(v)q_{k+2}(v)q_{k+3}(v)- 
q_k(v)q_{k+1}(v)q_{k+2}(v)q_{k+3}(v)q_{k+4}(v)=0  \label{eqq2} \\
 & 1- {\scriptstyle \frac{1}{|Q_k|^2}} q_k(v) 
+{\scriptstyle \frac{1}{|Q_k|^2|Q_{k-1}|^2}} q_k(v)q_{k-1}(v) 
-q_k(v)q_{k-1}(v)q_{k-2}(v)  \notag \\
&\ +{\scriptstyle\frac{1}{|Q_k|^2}} q_k(v)q_{k-1}(v)q_{k-2}(v)q_{k-3}(v)
- {\scriptstyle\frac{1}{|Q_k|^2|Q_{k-1}|^2} }
q_k(v)q_{k-1}(v)q_{k-2}(v)q_{k-3}(v)q_{k-4}(v)=0 \label{eqq3}
\end{align}
for $k=1,\dots,6$, where the indices are taken modulo $6$, and 
$|Q_k|=L_{k-1}^2/L_{k+1}^2$.
\end{lemma}

\begin{remark}\label{remqtof}
If the values of a function $q^\eps$ all lie in the upper half-plane $\{z\in\C 
: \text{Im}(z)>0\}$ and $\sum_{k=1}^6 \arg(q_k(v))=4\pi$ holds for all interior 
vertices $v$, then equations~\eqref{eqq2}--\eqref{eqq3} guarantee that 
$q^\eps$ corresponds to a discrete conformal map on (a part of) a triangular 
lattice $T_S$. Indeed, start with any triangle of $T_S$ and map it to any 
triangle in $\C$ respecting orientation. This defines the values of the 
discrete conformal map $g$ on the 
first triangle. Then the values of $g$ on all incident triangles can now be 
uniquely determined using~\eqref{eqdefq}. Our additional assumptions on 
$q^\eps$ show that the pattern is immersed. If $T_S$ is the triangulation of a 
simply connected 
domain, this procedure subsequently defines $g$ on all vertices. 
Equations~\eqref{eqq1}--\eqref{eqq3} guarantee that no ambiguities will occur.
\end{remark}

\begin{proof}[Proof of Lemma~\ref{lemeqq}.]
First note that~\eqref{eqq1} is an easy consequence of the fact that $q_k$ are 
cross-ratios of a flower. Furthermore, \eqref{eqq3} follows from~\eqref{eqq1} 
and~\eqref{eqq2} by taking the complex conjugation of~\eqref{eqq2} because 
$|q_k(v)|=|Q_k|$ using~\eqref{eqdefdiscf}.

In order to see~\eqref{eqq2},
add circumcircles to the triangles of the flower around $v_0$ and map $v_0$ to 
$\infty$ by a M\"obius transformation as illustrated in 
Figure~\ref{figFlowerMoeb} (right).
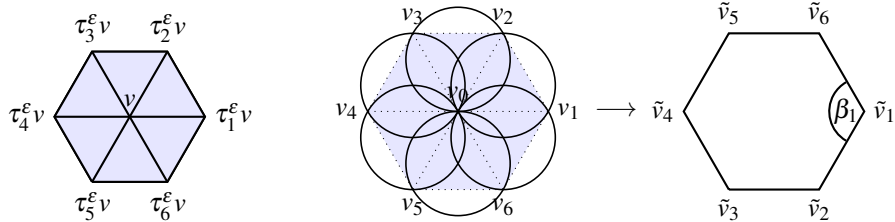
\begin{figure}[ht]
\begin{tikzpicture}[thick, scale=0.5]
  \coordinate [label={right:$\tau_1^\eps v$}] (B) at (2cm,0);
  \coordinate [label={above:$\tau_2^\eps v$}] (C) at (1cm, 1.732cm);
\coordinate [label={below:$\tau_6^\eps v$}] (D) at (1cm, -1.732cm);
\coordinate [label={left:$\tau_4^\eps v$}] (E) at (-2cm,0);
  \coordinate [label={above:$\tau_3^\eps v$}] (F) at (-1cm, 1.732cm);
\coordinate [label={below:$\tau_5^\eps v$}] (G) at (-1cm, -1.732cm);

\draw [thick] (G) -- (D) -- (B) -- (C) -- (F) -- (E) -- (G); 
\filldraw[fill=blue!10!white] (G) -- (D) -- (B) -- (C) -- (F) -- (E) -- (G);
 \coordinate [label={above:$v$}] (A) at (0, 0);
  \draw [thick] (A) -- (B) ;
\draw [thick] (A) -- (C) ;
\draw [thick] (A) -- (D) ;
\draw [thick] (A) -- (E) ;
\draw [thick] (A) -- (F) ;
\draw [thick] (A) -- (G) ;
\end{tikzpicture}
\hspace{2em}
\begin{tikzpicture}[scale=0.6]
  \coordinate [label={right:$v_1$}] (B) at (2cm,0);
  \coordinate [label={above:$v_2$}] (C) at (1cm, 1.732cm);
\coordinate [label={below:$v_6$}] (D) at (1cm, -1.732cm);
\coordinate [label={left:$v_4$}] (E) at (-2cm,0);
  \coordinate [label={above:$v_3$}] (F) at (-1cm, 1.732cm);
\coordinate [label={below:$v_5$}] (G) at (-1cm, -1.732cm);

\filldraw[dotted, fill=blue!10!white] (G) -- (D) -- (B) -- (C) -- (F) -- (E) -- 
(G);
 \coordinate [label={above:$v_0$}] (A) at (0, 0);
  \draw [dotted] (A) -- (B) ;
\draw  [dotted] (A) -- (C) ;
\draw [dotted] (A) -- (D) ;
\draw [dotted] (A) -- (E) ;
\draw [dotted] (A) -- (F) ;
\draw [dotted] (A) -- (G) ;

    \tkzCircumCenter(A,B,C)\tkzGetPoint{G1}
    \tkzDrawCircle(G1,A)
  \tkzCircumCenter(A,B,D)\tkzGetPoint{G2}
    \tkzDrawCircle(G2,A)
  \tkzCircumCenter(A,C,F)\tkzGetPoint{G3}
    \tkzDrawCircle(G3,A)
  \tkzCircumCenter(A,F,E)\tkzGetPoint{G4}
    \tkzDrawCircle(G4,A)
  \tkzCircumCenter(A,E,G)\tkzGetPoint{G5}
    \tkzDrawCircle(G5,A)
  \tkzCircumCenter(A,G,D)\tkzGetPoint{G6}
    \tkzDrawCircle(G6,A)

\draw (3.5cm,0) node {$\longrightarrow$};

 \coordinate [label={right:$\tilde{v}_1$}] (B1) at (9cm,0);
  \coordinate [label={above:$\tilde{v}_6$}] (C1) at (8cm, 1.732cm);
\coordinate [label={below:$\tilde{v}_2$}] (D1) at (8cm, -1.732cm);
\coordinate [label={left:$\tilde{v}_4$}] (E1) at (5cm,0);
  \coordinate [label={above:$\tilde{v}_5$}] (F1) at (6cm, 1.732cm);
\coordinate [label={below:$\tilde{v}_3$}] (G1) at (6cm, -1.732cm);
\draw [thick] (G1) -- (D1) -- (B1) -- (C1) -- (F1) -- (E1) -- (G1);

\draw (8.6cm,0) node {$\beta_1$};
\draw[thick]  (8.6cm,-0.65cm) arc (240:120:0.75cm);
\end{tikzpicture}
\caption{{\it Left:} A flower about $v$ in the lattice $\eps T$. {\it Right:} 
Mapping the circumcircles of a flower by a M\"obius 
transformation with $v_0\mapsto\infty$.}\label{figFlowerMoeb}
\end{figure}

As the M\"obius transformation does not change the values of the $q_k$'s, we 
easily 
identify equations~\eqref{eqq1} and~\eqref{eqq2} as the closing conditions for 
the image polygon.
\end{proof}


Our next goal is to derive from~\eqref{eqq1}--\eqref{eqq3} an expression for 
the Laplacian of the discrete Schwarzians
$\Delta^\eps s_k(v)$ which equals a polynomial in $\eps, s_j(v), \tau_l 
s_m(v)$ for $j,l,m\in\{1,\dots,6\}$. Then, if all discrete Schwarzians $s_k$ 
are uniformly 
bounded for small $\eps\in(0,\eps_0)$ and on $V_K^\eps$, also all Laplacians 
$\Delta^\eps s_k(v)$ are uniformly bounded on $(V_K^\eps)^{(1)}$.

For an interior vertex $v$ substitute $q_k(v)=Q_k(1+\eps^2s_k(v))$ 
in~\eqref{eqq1}--\eqref{eqq3}, and obtain (using for example a computer algebra 
program):
\begin{align}
 \sum_{k=1}^6 s_k(v) =&\eps^2 \Phi(v)\label{eqPhi}
\end{align}
\begin{multline}
s_k(v) + s_{k+1}(v) + s_{k+2}(v)- Q_k (s_{k+1}(v)+s_{k+2}(v) + s_{k+3}(v)) 
\\
  +Q_kQ_{k+1} (s_{k+2}(v) +s_{k+3}(v) +s_{k+4}(v))
     =\eps^2\Psi_k(v) \label{eqPsi}
\end{multline}
\begin{multline}
s_k(v) + s_{k-1}(v) + s_{k-2}(v)-  \frac{1}{\overline{Q_k}} 
(s_{k-1}(v)+s_{k-2}(v) + s_{k-3}(v)) \\
  +\frac{1}{\overline{Q_k}\overline{Q_{k-1}}} (s_{k-2}(v) 
+s_{k-3}(v) +s_{k-4}(v)) =\eps^2\Theta_k(v) \label{eqTheta}
\end{multline}

\begin{align*}
\text{ where }\qquad\qquad
\Phi(v)=& - \sum_{k=1}^5\sum_{l=k+1}^6 s_k(v)s_l(v) 
- \eps^2\sum_{\substack{k_1,k_2,k_3\in\{1,\dots,6\} \\ k_1<k_2<k_3}} 
s_{k_1}(v)s_{k_2}(v)s_{k_3}(v) \notag \\
 &\quad - \eps^4\sum_{\substack{k_1,k_2,k_3,k_4\in\{1,\dots,6\} \\ 
k_1<k_2<k_3<k_4}} 
s_{k_1}(v)s_{k_2}(v)s_{k_3}(v)s_{k_4}(v) \notag \\
&\quad - \eps^6\sum_{\substack{k_1,k_2,k_3,k_4,k_5\in\{1,\dots,6\} \\ 
k_1<k_2<k_3<k_4<k_5}} 
s_{k_1}(v)s_{k_2}(v)s_{k_3}(v)s_{k_4}(v)s_{k_5}(v) \notag \\
&\quad - \eps^8 s_1(v) s_2(v) s_3(v) s_4(v) s_5(v) s_6(v), \notag
\end{align*}

\begin{align*}
\Psi_k(v)&= -  \sum_{\substack{k_1,k_2\in\{k,\dots,k+2\} \\ k_1<k_2}} 
s_{k_1}(v)s_{k_2}(v) 
+Q_k\sum_{\substack{k_1,k_2\in\{k,\dots,k+3\} \\ k_1<k_2}} 
s_{k_1}(v)s_{k_2}(v) \notag \\ 
&\qquad\qquad -Q_kQ_{k+1}\sum_{\substack{k_1,k_2\in\{k,\dots,k+4\} \\ k_1<k_2}} 
s_{k_1}(v)s_{k_2}(v) \notag \\
&\quad +\eps^2\left( s_k(v) s_{k+1}(v) s_{k+2}(v) - 
Q_k \sum_{\substack{k_1,k_2,k_3\in\{k,\dots,k+3\} \\ k_1<k_2<k_3}} 
s_{k_1}(v)s_{k_2}(v)s_{k_3}(v)\right. \notag \\
&\qquad\qquad\left. +Q_kQ_{k+1} \sum_{\substack{k_1,k_2,k_3\in\{k,\dots,k+4\} 
\\ 
k_1<k_2<k_3}} s_{k_1}(v)s_{k_2}(v)s_{k_3}(v)\right) \notag \\
&\quad +\eps^4\left(-Q_k s_k(v) s_{k+1}(v) s_{k+2}(v) s_{k+3}(v) \right. \notag 
\\
&\qquad\qquad\left. +Q_kQ_{k+1} 
\sum_{\substack{k_1,k_2,k_3,k_4\in\{k,\dots,k+4\} \\ 
k_1<k_2<k_3<k_4}} s_{k_1}(v)s_{k_2}(v)s_{k_3}(v)s_{k_4}(v)\right) \notag \\
&\quad +\eps^6 Q_kQ_{k+1} s_k(v) s_{k+1}(v) s_{k+2}(v) s_{k+3}(v) s_{k+4}(v), 
\notag
%
\end{align*}

\begin{align*}
\Theta_k(v)&= -\sum_{\substack{k_1,k_2\in\{k-2,\dots,k\} \\ k_1<k_2}} 
s_{k_1}(v)s_{k_2}(v) 
+\frac{1}{\overline{Q_k}} \sum_{\substack{k_1,k_2\in\{k-3,\dots,k\} \\ 
k_1<k_2}} 
s_{k_1}(v)s_{k_2}(v) \notag \\
&\qquad\qquad -\frac{1}{\overline{Q_k}\overline{Q_{k-1}}} 
\sum_{\substack{k_1,k_2\in\{k-4,\dots,k\} \\ k_1<k_2}} 
s_{k_1}(v)s_{k_2}(v) \notag \\
&\quad +\eps^2\left(s_k(v) s_{k-1}(v) s_{k-2}(v) -
\frac{1}{\overline{Q_k}} \sum_{\substack{k_1,k_2,k_3\in\{k-3,\dots,k\} \\ 
k_1<k_2<k_3}} s_{k_1}(v)s_{k_2}(v)s_{k_3}(v) \right. \notag \\
&\qquad\qquad \left.+\frac{1}{\overline{Q_k}\overline{Q_{k-1}}} 
\sum_{\substack{k_1,k_2,k_3\in\{k-4,\dots,k\} \\ 
k_1<k_2<k_3}} s_{k_1}(v)s_{k_2}(v)s_{k_3}(v)\right) \notag \\
&\quad +\eps^4\left(-\frac{1}{\overline{Q_k}} s_k(v) s_{k-1}(v) s_{k-2}(v) 
s_{k-3}(v) \right. \\
&\qquad\qquad +\left. \frac{1}{\overline{Q_k}\overline{Q_{k-1}}} 
\sum_{\substack{k_1,k_2,k_3,k_4\in\{k-4,\dots,k\} \\ 
k_1<k_2<k_3<k_4}} s_{k_1}(v)s_{k_2}(v)s_{k_3}(v)s_{k_4}(v)\right) \notag \\
&\quad +\eps^6 \frac{1}{\overline{Q_k}\overline{Q_{k-1}}} s_k(v) s_{k-1}(v) 
s_{k-2}(v) s_{k-3}(v) s_{k-4}(v) \notag
\end{align*}

\begin{lemma}\label{lemDeltas}
 $\Delta^\eps s_k(v)$ is equal to a polynomial in the variables $\eps, s_j(v),
 \tau_m^\eps s_l(v)$, $j,m,l\in\{1,\dots,6\}$. In particular,
 \begin{align*}
\Delta^\eps s_k = {\frac{|Q_k|}{4i L_k^2Q_k}} \Big( &
  \overline{Q_{k+1}}
(\tau_{k-1}^\eps\Psi_k+ \tau_{k+1}^\eps\Psi_{k+3})
  +Q_1Q_2
(\tau_{k-1}^\eps\Theta_{k+3} +\tau_{k+1}^\eps\Theta_k) \\
  &-{\textstyle \frac{L_k^2}{L_{k+1}^2L_{k-1}^2}} 
(\tau_{k-1}^\eps\Phi +\tau_{k+1}\Phi)
+ ({\textstyle \frac{1}{\overline{Q_k}}} -1) 
(\tau_k^\eps\Psi_k +\Psi_{k+3}) \\[0.5ex]
& \left. +(1-Q_k)(\tau_k^\eps\Theta_{k+3} +\Theta_k)
  + \overline{Q_{k+1}}(Q_k-1) (\tau_k^\eps\Phi+\Phi)  \right) ,
\end{align*}
 where the indices are taken modulo $6$. 
\end{lemma}
\begin{proof}
 We consider the case $k=1$. Fix $v_0\in (V_K^\eps)^{(2)}$. Denote 
$v_1=\tau_1^\eps v_0$,
 $v_2=\tau_2^\eps v_0$ and $v_6=\tau_6^\eps v_0$.
 
 First recall that $s_{4}(\tau_4^\eps v)=s_1(v)$ and $s_{1}(\tau_1^\eps 
v)=s_4(v)$. Thus $\Delta^\eps s_1(v)$ only involves the values of $s_1$ and 
$s_4$ at $v_0,v_1,v_2,v_6$.
 
 Take $(-\frac{1}{4i}\frac{\sin^2\alpha}{\sin^2\beta \sin^2\gamma})$ 
 times~\eqref{eqPhi}
 and add $\frac{1}{4i}\frac{\text{e}^{2i\beta}}{\sin^2\beta} 
=\frac{\overline{Q_{2}}|Q_1|}{4i L_1^2Q_1}$ times~\eqref{eqPsi} for $k=1$
 and $\frac{1}{4i}\frac{\text{e}^{2i\gamma}}{\sin^2\gamma} 
=\frac{Q_1Q_{2}|Q_1|}{4i L_1^2Q_1}$ times~\eqref{eqTheta} for $k=4$.
 After simplification we obtain
 \begin{multline*}
  \cot\beta\, s_1(v_6)+\cot\gamma\; s_4(v_6) +  
(\cot\beta+\cot\gamma)(s_2(v_6)+s_3(v_6)) \\
  = \frac{\eps^2}{4i}\left( \frac{\text{e}^{2i\beta}}{\sin^2\beta} \Psi_1(v_6)
  +\frac{\text{e}^{2i\gamma}}{\sin^2\gamma}\Theta_4(v_6)
  -\frac{\sin^2\alpha}{\sin^2\beta \sin^2\gamma} \Phi(v_6)\right).
 \end{multline*}
By cyclic permutation we also have
\begin{multline*}
  \cot\gamma\; s_1(v_2)+\cot\beta\, s_4(v_2) +   
(\cot\beta+\cot\gamma)(s_5(v_2)+s_6(v_2)) \\
  = \frac{\eps^2}{4i}\left( \frac{\text{e}^{2i\beta}}{\sin^2\beta} \Psi_4(v_2)
  +\frac{\text{e}^{2i\gamma}}{\sin^2\gamma}\Theta_1(v_2)
  -\frac{\sin^2\alpha}{\sin^2\beta \sin^2\gamma} \Phi(v_2)\right).
 \end{multline*}
 Now combine $\frac{1}{4i}(\frac{\text{e}^{2i\gamma}}{\sin^2\gamma}
 -\frac{\text{e}^{2i\beta}}{\sin^2\beta})= 
\frac{(Q_1-1)\overline{Q_{2}}|Q_1|}{4i L_1^2Q_1}$  times~\eqref{eqPhi} with 
$\frac{1}{4i}(\frac{1}{\sin^2\beta}
 +2\cot\alpha\frac{\text{e}^{i\beta}}{\sin^2\beta}) 
=\frac{(1/\overline{Q_1} -1)|Q_1|}{4i L_1^2Q_1}$ times~\eqref{eqPsi} 
 for $k=1$ and $\frac{1}{4i}(-\frac{1}{\sin^2\gamma}
 -2\cot\alpha\frac{\text{e}^{i\gamma}}{\sin^2\gamma}) =\frac{(1-Q_1) 
|Q_1|}{4i L_1^2Q_1}$ times~\eqref{eqTheta} for $k=4$. This gives
 \begin{multline*}
  \cot\alpha\; s_1(v_1)-(\cot\alpha +\cot\beta +\cot\gamma) s_4(v_1) - 
    (\cot\beta+\cot\gamma)(s_3(v_1)+s_5(v_1)) \\
  = {\textstyle \frac{\eps^2}{4i}}\left( ({\textstyle \frac{1}{\sin^2\beta}}
 + {\textstyle \frac{2\text{e}^{i\beta}\cot\alpha}{\sin^2\beta}})\Psi_1(v_1)
  -({\textstyle \frac{1}{\sin^2\gamma}}
 +{\textstyle  \frac{2\text{e}^{i\gamma}\cot\alpha}{\sin^2\gamma}})\Theta_4(v_1)
  + ({\textstyle \frac{\text{e}^{2i\gamma}}{\sin^2\gamma}}
 -{\textstyle \frac{\text{e}^{2i\beta}}{\sin^2\beta}})\Phi(v_1)\right).
 \end{multline*}
 By cyclic permutation we also have
  \begin{multline*}
  \cot\alpha\; s_4(v_0)-(\cot\alpha +\cot\beta +\cot\gamma) s_1(v_0) - 
    (\cot\beta+\cot\gamma)(s_2(v_0)+s_6(v_0)) \\
  = {\textstyle \frac{\eps^2}{4i}}( ({\textstyle \frac{1}{\sin^2\beta}}
 + {\textstyle \frac{2\text{e}^{i\beta}\cot\alpha}{\sin^2\beta}})\Psi_4(v_0)
  -({\textstyle \frac{1}{\sin^2\gamma}}
 + {\textstyle \frac{2\text{e}^{i\gamma}\cot\alpha}{\sin^2\gamma}})\Theta_1(v_0)
  + ({\textstyle \frac{\text{e}^{2i\gamma}}{\sin^2\gamma}}
 -{\textstyle \frac{\text{e}^{2i\beta}}{\sin^2\beta}})\Phi(v_0)).
 \end{multline*}
 Adding up these four equations and dividing by $\eps^2$ we finally
 arrive at
 \begin{align*}
  \Delta^\eps s_1(v)= {\textstyle \frac{1}{4i}}&  \Big( 
  {\textstyle \frac{\text{e}^{2i\beta}}{\sin^2\beta}} (\Psi_1(v_6)+\Psi_4(v_2))
  +{\textstyle \frac{\text{e}^{2i\gamma}}{\sin^2\gamma}} 
(\Theta_4(v_6)+\Theta_1(v_2)) \\
  &-{\textstyle \frac{\sin^2\alpha}{\sin^2\beta \sin^2\gamma}} 
(\Phi(v_6)+\Phi(v_2)) +\left({\textstyle \frac{1}{\sin^2\beta}}
 +{\textstyle \frac{2 \text{e}^{i\beta}\cot\alpha}{\sin^2\beta}}\right) 
(\Psi_1(v_1)+\Psi_4(v_0)) \\
& \left. -\left({\textstyle \frac{1}{\sin^2\gamma}}
 +{\textstyle \frac{2 \text{e}^{i\gamma}\cot\alpha}{\sin^2\gamma}}\right) 
(\Theta_4(v_1) +\Theta_1(v_0))
  + \left({\textstyle \frac{\text{e}^{2i\gamma}}{\sin^2\gamma}}
 -{\textstyle \frac{\text{e}^{2i\beta}}{\sin^2\beta}}\right) 
(\Phi(v_1)+\Phi(v_0))  \right)
 \end{align*}
Again, we have used~\eqref{eqsk+}. This proves the lemma for $k=1$. For other 
values of $k$ the lemma is also true by symmetry.
\end{proof}

\section{Boundedness and convergence of the discrete 
Schwarzians}\label{SecBound}

We start by showing that we used a suitable order of $\eps$ in the definition 
of the discrete Schwarzians, as they are bounded in the limit $\eps\to 0$.

\begin{lemma}\label{lembound}
 Let $v_0\in V_K^\eps$ be an interior vertex. Then for $\eps$ small enough
\begin{equation}\label{eqskbound}
|s_k(v_0)|=\eps^{-2}|q_k(v_0)/Q_k -1|\leq C
\end{equation}
for some constant $C$, which depends only on $\alpha,\beta,\gamma$, $K$ and  
$f$.
\end{lemma}
\begin{proof}
By our assumptions, the maps $f^\eps$ are discrete conformal. This means in 
particular by Definition~\ref{defdcm}, that the absolute values of the 
cross-ratios $q_k$ and $Q_k$ agree. Thus 
\begin{equation}\label{eqskIm}
 s_k(v_0) =\frac{1}{\eps^{2}}(\text{e}^{i(\arg q_k-\arg Q_k)}-1)= 
\frac{2i}{\eps^{2}}\sin((\arg q_k-\arg Q_k)/2)\text{e}^{i(\arg 
q_k-\arg Q_k)/2}. 
\end{equation}
Recall that $\arg Q_k = \arg (\omega_{k-1}^2/\omega_{k+1}^2)$. If we take $\arg 
Q_k=:\phi_k\in(0,2\pi)$, then $\phi_k$ is the sum of the two 
angles opposite to the edge $\omega_kL_k$ in the triangles in the lattice $T$
(and in fact $\phi_k\in (0,\pi)$ as we assumed strictly acute angles). 
Similarly, the angle $\arg q_k=:\varphi_k(v_0)\in(0,2\pi)$ is the interior 
intersection angle of the 
circumcircles of the two image triangles used for the cross-ratio $q_k$. In 
particular, $\varphi_k(v_0)$ is the sum of the two angles in the image pattern 
$f^\eps(T_K^\eps)$ opposite to the edge $f^\eps([v_0,v_0+\eps\omega_kL_k])$. 
Thus~\eqref{eqskbound} is satisfied if we can show that $|\varphi_k(v_0)- 
\phi_k| \leq \hat{C}\eps^2$ holds for some constant $\hat{C}$ which is 
independent of $v_0$. 
For the calculation of the angles in the triangles, we use the following 
half-angle formula
\begin{equation}\label{defalpha1}
 \tan\left(\frac{\alpha}{2}\right) = 
\sqrt{\frac{(-b+a+c)(-c+a+b)}{(b+c-a)(a+b+c)}}
\end{equation}
with the notation of Figure~\ref{figTriang}.
%
The discrete conformality~\eqref{eqdefdiscf} together with~\eqref{defalpha1} 
implies that
\begin{align*}
\frac{\varphi_k(v_0)}{2} &= \arctan\left( 
\sqrt{\frac{(L_k\text{e}^{-d_{k+1}} +L_{k+1}\text{e}^{-d_{k}} 
-L_{k-1}) (L_k\text{e}^{-d_{k+1}} -L_{k+1}\text{e}^{-d_{k}} 
+L_{k-1})}{(-L_k\text{e}^{-d_{k+1}} +L_{k+1}\text{e}^{-d_{k}} 
+L_{k-1}) (L_k\text{e}^{-d_{k+1}} +L_{k+1}\text{e}^{-d_{k}} 
+L_{k-1})}} \right) \\
&\ \ +\arctan\left( \sqrt{\frac{(L_k\text{e}^{-d_{k-1}} 
+L_{k-1}\text{e}^{-d_{k}} -L_{k+1}) (L_k\text{e}^{-d_{k-1}} 
-L_{k-1}\text{e}^{-d_{k}} 
+L_{k+1})}{(-L_k\text{e}^{-d_{k-1}} +L_{k-1}\text{e}^{-d_{k}} 
+L_{k+1}) (L_k\text{e}^{-d_{k-1}} +L_{k-1}\text{e}^{-d_{k}} 
+L_{k+1})}} \right),
\end{align*}
where we denote by $d_j=(u(v_0+\eps\omega_jL_j)- u(v_0))/2$ for $j=k-1,k,k+1$ 
the differences of the logarithmic scale factors. 
Then~\eqref{eqconvu} and the uniform convergence of $f^\eps$ implies 
that for $\eps>0$ small enough we can write
\begin{align*}
\frac{\varphi_k(v_0)}{2} &= \arctan\left( 
\sqrt{\frac{(L_k\text{e}^{-g_{k+1}} +L_{k+1}\text{e}^{-g_{k}} 
-L_{k-1}) (L_k\text{e}^{-g_{k+1}} -L_{k+1}\text{e}^{-g_{k}} 
+L_{k-1})}{(-L_k\text{e}^{-g_{k+1}} +L_{k+1}\text{e}^{-g_{k}} 
+L_{k-1}) (L_k\text{e}^{-g_{k+1}} +L_{k+1}\text{e}^{-g_{k}} 
+L_{k-1})}} \right) \\
&\ \ +\arctan\left( \sqrt{\frac{(L_k\text{e}^{-g_{k-1}} 
+L_{k-1}\text{e}^{-g_{k}} -L_{k+1}) (L_k\text{e}^{-g_{k-1}} 
-L_{k-1}\text{e}^{-g_{k}} 
+L_{k+1})}{(-L_k\text{e}^{-g_{k-1}} +L_{k-1}\text{e}^{-g_{k}} 
+L_{k+1}) (L_k\text{e}^{-g_{k-1}} +L_{k-1}\text{e}^{-g_{k}} 
+L_{k+1})}} \right) \\
&\ \ + {\cal O}(\eps^2),
\end{align*}
where we denote $g_j= (\log|f'(v_0+\eps\omega_jL_j)|-\log|f'(v_0)|)/2$ for 
$j=k-1, k,k+1$. The notation  $h(\eps)={\cal O}(\eps^n)$ means that there is a
constant $\cal C$, such that $|h(\eps)|\leq {\cal C}\eps^n$ holds for all small 
enough $\eps>0$. Note that the constant is independent of $v_0$ due to 
estimate~\eqref{eqconvu}. Now a Taylor expansion in $\eps$ (using for example a 
computer algebra program) shows that
\begin{align*}
 \frac{\varphi_k(v_0)}{2}&= \frac{\phi_k}{2} + {\cal O}(\eps^2).
\end{align*}
This implies that $|\varphi_k- \phi_k| \leq \hat{C}\eps^2$ for some constant 
$\hat{C}$ (independent of $v_0$) and finally~\eqref{eqskbound} for $\eps$ small 
enough.
\end{proof}

The following lemma on regularity of solutions of discrete elliptic equations 
constitutes another main ingredient for our convergence proof.

\begin{lemma}[Regularity lemma]\label{lemReg}
 Let $W$ be a subset of $V^\eps$. Let $v_0\in W^{(1)}$ and let $\delta$ be the 
Euclidean distance from $v_0$ to $V^\eps\setminus W$. Let $\eta:W\to\R$ be any 
function. Then
\[\delta |\partial_k^\eps\eta(v_0)|\leq 7 \|\eta\|_W + 
\frac{\delta^2}{C_{\alpha\beta\gamma}}\|\Delta^\eps\eta\|_{W^{(1)}}\]
holds for $k=1,\dots,6$, where $C_{\alpha\beta\gamma}=\Delta^\eps 
x^2=\Delta^\eps y^2 = (\sin(2\alpha)+\sin(2\beta) +\sin(2\gamma))/2$. 
\end{lemma}
Note that this lemma is a version of~\cite[Regularity Lemma~7.1]{HeSch98} and 
we leave the small, but necessary adaptions of the proof to the reader.

Using this Regularity lemma we will deduce convergence of a subsequence of the 
discrete conformal maps of Theorem~\ref{theoConv} as we already know from 
Lemma~\ref{lembound} that the discrete Schwarzians $s_k$ are uniformly bounded 
with a bound independent of $\eps$. Lemma~\ref{lemDeltas} now implies that the 
functions $\Delta^\eps s_k$ are also uniformly bounded. By Lemma~\ref{lemReg} 
also $\partial_j^\eps s_k$ has such a bound (locally uniformly). Thus it 
follows by 
Lemma~\ref{lemconv} that for some sequence of $\eps$ tending to $0$ there 
exists a continuous limit ${\cal S}_k=\lim\limits_{\eps\to 0} s_k$, which are 
Lipschitz functions on the interior of $K$. Note that relation~\eqref{eqsk+} 
implies that
\begin{equation}
 {\cal S}_{k+3}={\cal S}_k.
\end{equation}
Together with~\eqref{eqPhi} this gives
\begin{equation}\label{eqsumSk}
 {\cal S}_{1}+{\cal S}_2+{\cal S}_3=0.
\end{equation}

For simplicity, we assume that the boundary $\partial K$ of the compact set $K$ 
in Theorem~\ref{theoConv} has {\em positive 
reach} $\geq {\cal R}>0$. This means that for every $0<\delta<{\cal R}$ all 
points $p\in\C$ with distance $\text{d}(p,\partial K)$ at most $\delta$ to the 
boundary $\partial K$ have a unique projection onto $\partial K$, that is 
there exists a unique point $x\in\partial K$ such that 
$|x-p|=\text{d}(p,\partial K)$. Any compact set $K$ which is the 
closure of its simply connected interior can be approximated by such compact 
sets with positive reach.

For every $\delta>0$ denote by $V_{K,\delta}^\eps$ the vertices of $V_K^\eps$ 
which have at least Euclidean distance $\delta$ to any vertex in 
$V^\eps\setminus V_K^\eps$.
As $\partial K=\partial \Omega$ is assumed to have positive reach $\geq \cal R$ 
and as $\Omega$ is simply connected, $V_K^\eps$ contains all vertices 
whose distance to the boundary $\partial K$ is at least $\eps$ for all 
$\eps<{\cal R}/4$.

\begin{lemma}\label{lempartialbound}
Let $n\in\N$ and ${\cal R}/4>\delta>0$. There are constants $C=C(n,\delta)$ and 
$\mu_n>0$ such that
\[ \|\partial_{j_n}^\eps \partial_{j_{n-1}}^\eps \dots \partial_{j_1}^\eps 
s_k\|_{(V_{K,\delta}^\eps)^{(n)}} \leq C\]
holds whenever $\eps<\mu_n$ and $k, j_1,\dots,j_n\in\{1,\dots,6\}$. In other 
words, the functions $s_k$ are uniformly bounded in $C^\infty(\Omega)$.
\end{lemma}

The proof is very similar to the proof of Lemma~8.1 in~\cite{HeSch98}.
\begin{proof}
 We use induction on $n$. The case $n=0$ has been shown in 
Lemma~\ref{lembound}. 

So let $n>0$ and assume that the lemma holds for $0,\dots,n-1$. 
Consider $g=\partial_{j_{n-1}}^\eps \dots \partial_{j_1}^\eps s_k$. Then 
Lemma~\ref{lemDeltas} implies that $\Delta^\eps g=\Delta^\eps 
\partial_{j_{n-1}}^\eps \dots \partial_{j_1}^\eps s_k =\partial_{j_{n-1}}^\eps 
\dots \partial_{j_1}^\eps\Delta^\eps s_k$ is a linear combination of functions 
of the form $\partial_{j_{n-1}}^\eps \dots \partial_{j_1}^\eps F$ where $F$ is 
one of the functions $\tau^\eps_{m_1} \Phi$, $\tau^\eps_{m_1} \Psi_{m_2}$, 
$\tau^\eps_{m_1} \Theta_{m_2}$, $\Phi$, $\Psi_{m_2}$, $\Theta_{m_2}$ for 
$m_1,m_2=1,\dots,6$. Recall 
from~\eqref{eqPhi}--\eqref{eqTheta} that these functions are polynomials in 
$\eps$ and the $s_l$'s. From the product rule~\eqref{prodrule} it follows by 
induction that $\partial_{j_{n-1}}^\eps \dots \partial_{j_1}^\eps F$ is also 
a polynomial in $\eps$ and expressions of the form $\tau^\eps_{j_m}\dots 
\tau^\eps_{j_{s+1}}\partial^\eps_{j_s}\dots\partial_{j_1}s_l$ where $m\leq 
n-1$. 

Let $v\in V_{K,\delta/2}^\eps$, $4n\eps<\delta$. Then 
$v'=\tau^\eps_{j_m}\dots \tau^\eps_{j_{s+1}}v \in (V_{K,\delta/4}^\eps)^{(s)}$ 
if $m\leq n$.
Now the induction hypothesis with $v'=\tau^\eps_{j_m}\dots \tau^\eps_{j_{s+1}} 
v$, $n'=s$ and $\delta'=\delta/4$ applies and provides a bound for 
$\tau^\eps_{j_m}\dots \tau^\eps_{j_{s+1}}\partial^\eps_{j_s}\dots\partial_{j_1}
s_l$ at $v$. Since $\Delta^\eps g$ is a polynomial in $\eps$ and the 
expressions of the form $\tau^\eps_{j_m}\dots 
\tau^\eps_{j_{s+1}}\partial^\eps_{j_s} \dots\partial_{j_1} s_l$ for $m\leq n$ 
and $s\leq n-1$, we deduce that 
$\|\Delta^\eps g\|_{V_{K,\delta/2}^\eps}\leq C_1$ for some constant 
$C_1=C_1(\delta)$. As $|g|$ is also bounded on 
$V_{K,\delta/2}^\eps$ by the induction hypothesis, the Regularity 
lemma~\ref{lemReg} implies that $\partial_{k_{n}}^\eps g$ is bounded on 
$V_{K,\delta}^\eps$ and therefore the induction step holds. This finishes the 
proof. 
\end{proof}

\begin{corollary}\label{corconvsk}
 $s_k\to {\cal S}_k$ in $C^\infty(\Omega)$ as $\eps\to 0$.
\end{corollary}
\begin{proof}
By Lemma~\ref{lempartialbound} and Lemma~\ref{lemconv} this claim is true for 
some subsequence. The general statement will follow later as corollary of 
Theorem~\ref{theoConvS}, when we prove the convergence in full generality.
\end{proof}

\section{$C^\infty$-convergence of the discrete conformal maps 
$f^\eps$}\label{SecConv}

As primary step to the convergence of the discrete conformal maps $f^\eps$ we 
consider special M\"obius transformations from triangles of $T_K^\eps$ to their 
images under $f^\eps$. In particular, define the contact transformation 
$Z^\eps_k =Z^\eps_k(v)$ for any interior vertex $v$ to be the M\"obius 
transformation which maps the three points $0$, $\eps L_k\omega_k$, $\eps 
L_{k+1}\omega_{k+1}\in V^\eps$ to the three points $f^\eps(v)$, 
$f^\eps(\tau^\eps_k v)$, 
$f^\eps(\tau^\eps_{k+1} v)$ respectively.

Let $R_k^\eps$ denote the translation $R_k^\eps(z)=z+\eps L_k\omega_k$. Then we 
easily note that
\begin{equation*}
 Z_{k+2}^\eps(\tau_k^\eps v)= Z_k^\eps(v)\cdot R_k^\eps.
\end{equation*}
Furthermore, we have
\begin{equation*}
 Z_{k-1}^\eps(v)= Z_k^\eps(v)\cdot M_k^\eps(v),\qquad \text{where }
M_k^\eps(v)= \frac{L_k\omega_k(1+\eps^2s_k(v))z}{\eps s_k(v)z+L_k\omega_k}.
\end{equation*}
These two relations allow us to express $\tau_k^\eps Z_k^\eps(v)$ in terms of 
$Z_k^\eps(v)$ and the transition matrices $R_k^\eps$ and $M_k^\eps(v)$:
\begin{equation}\label{eqtauZ}
 \tau_k^\eps Z_k=\tau_k^\eps(Z_{k+2}^\eps\cdot M_{k+2}^\eps \cdot 
M_{k+1}^\eps) =Z_k^\eps\cdot R_k^\eps\cdot \tau_k^\eps M_{k+2}^\eps \cdot 
\tau_k^\eps M_{k+1}^\eps.
\end{equation}
The matrix representations of $R_k^\eps$ and $M_k^\eps$ are
\begin{equation*}
 R_k^\eps=\begin{pmatrix} 1 & \eps L_k\omega_k \\[0.5ex] 0 & 1 
\end{pmatrix}\qquad \text{and}\qquad
M_k^\eps= \begin{pmatrix} 1+\eps^2s_k &0 \\[0.5ex] \frac{\eps s_k}{L_k\omega_k} 
&1\end{pmatrix}.
\end{equation*}
Note that both $R_k^\eps$ and $M_k^\eps$ are polynomial in $\eps$ and $s_k$.
Now direct computation gives
\begin{equation*}
 R_k^\eps\cdot \tau_k^\eps M_{k+2}^\eps \cdot \tau_k^\eps M_{k+1}^\eps=
I+\eps \begin{pmatrix} 0 & L_k\omega_k \\ \frac{\tau_k^\eps 
s_{k+1}}{L_{k+1}\omega_{k+1}} +\frac{\tau_k^\eps s_{k+2}}{L_{k+2}\omega_{k+2}} 
&0 \end{pmatrix} + \eps^2 {\cal O}(1),
\end{equation*}
where $I$ is the identity matrix and ${\cal O}(1)$ denotes some matrix which is 
polynomial in $\eps,\tau_k^\eps s_{k+1}, \tau_k^\eps s_{k+2}$. Combined 
with~\eqref{eqtauZ} the discrete derivative $\partial_k^\eps Z_k^\eps$ may be 
express as follows:
\begin{equation}\label{eqdkZ}
 \partial_k^\eps Z_k^\eps= Z_k^\eps\cdot \begin{pmatrix} 0 & \omega_k \\ 
\frac{\tau_k^\eps s_{k+1}}{L_kL_{k+1}\omega_{k+1}} +\frac{\tau_k^\eps 
s_{k+2}}{L_kL_{k+2}\omega_{k+2}} &0 \end{pmatrix}+\eps Z_k\cdot {\cal O}(1).
\end{equation}
Similar computations give
\begin{equation}\label{eqdk+1Z}
 \partial_{k+1}^\eps Z_k^\eps= Z_k^\eps\cdot \begin{pmatrix} 0 & \omega_{k+1} 
\\ 
-\frac{\tau_{k+1}^\eps s_{k}}{L_{k+1}L_k\omega_{k}} -\frac{\tau_{k+1}^\eps 
s_{k-1}}{L_{k+1}L_{k-1}\omega_{k-1}} &0 \end{pmatrix} +\eps Z_k\cdot {\cal 
O}(1).
\end{equation}
(Of course, the matrix in ${\cal O}(1)$ differs in general from the one 
in~\eqref{eqdkZ}.)
As $\partial_{k+3}^\eps = -\partial_k^\eps \tau_{k+3}^\eps$ and 
$\partial_{k+2}^\eps= \frac{L_{k+1}}{L_{k+2}} \partial_{k+1}^\eps 
\tau_{k+3}^\eps -\frac{L_{k+3}}{L_{k+2}} \partial_k^\eps \tau_{k+3}^\eps$ the 
expressions for all derivatives $\partial_j^\eps Z_k^\eps$ can be obtained 
from~\eqref{eqdkZ} and~\eqref{eqdk+1Z}.

Recall that $0$ is a vertex of $V_K^\eps$
and define $\hat{Z}_k^\eps(v)= Z_k^\eps(0)^{-1} 
Z_k^\eps(v)$. Then we deduce from~\eqref{eqdkZ} and~\eqref{eqdk+1Z} and the 
similar expressions for the other $\tau_j^\eps Z_k$ that 
$\tau_j^\eps\hat{Z}_k^\eps =\hat{Z}_k^\eps(I+\eps {\cal O}(1))$. As 
$\hat{Z}_k^\eps(0)=I$ and as $T_K^\eps$ is part of a 
(scaled) lattice contained in the compact set $K$ we deduce that 
$\hat{Z}_k^\eps$ is bounded, independently of $\eps$. From the corresponding 
relations for $Z_k^\eps$ we deduce that for $j=1,\dots,6$,
\begin{equation}\label{eqdhatZ}
 \partial_j^\eps\hat{Z}_k^\eps=\hat{Z}_k^\eps\cdot {\cal O}(1).
\end{equation}
As the elements of the matrix in the ${\cal O}(1)$-term are polynomials in 
$\eps$ and $\tau_l^\eps s_m$ for $m,l=1,\dots, 6$, 
Lemma~\ref{lempartialbound} implies that the ${\cal O}(1)$-term is bounded in 
$C^\infty(\Omega)$. Therefore, repeated differentiation of~\eqref{eqdhatZ} 
shows 
that $\hat{Z}_k^\eps$ is bounded in $C^\infty(\Omega)$. Now we can again deduce 
from 
Lemma~\ref{lemconv} that along some subsequence $\eps\to 0$ the limit 
$\hat{\cal Z}_k=\lim\limits_{\eps\to 0} \hat{Z}_k^\eps$ exists and that the 
convergence is in $C^\infty(\Omega)$. Moreover, equations~\eqref{eqdkZ} 
and~\eqref{eqdk+1Z} imply
\begin{align}
 \partial_k \hat{\cal Z}_k &= \hat{\cal Z}_k\cdot \begin{pmatrix} 0 & \omega_k 
\\ 
\frac{ {\cal S}_{k+1}}{L_kL_{k+1}\omega_{k+1}} +\frac{ 
{\cal S}_{k+2}}{L_kL_{k+2}\omega_{k+2}} &0\end{pmatrix}, \label{eqdkcalZa} \\
\partial_{k+1} \hat{\cal Z}_k &= \hat{\cal Z}_k\cdot \begin{pmatrix} 0 
& \omega_{k+1} \\ -\frac{ {\cal S}_{k}}{L_{k+1}L_k\omega_{k}} 
-\frac{ {\cal S}_{k-1}}{L_{k+1}L_{k-1}\omega_{k-1}} &0\end{pmatrix}. 
\label{eqdkcalZb}
\end{align}
These relations show that $\partial_{k+1} \hat{\cal Z}_k= 
\frac{\omega_{k+1}}{\omega_k}\partial_k \hat{\cal Z}_k$, which means that 
$\hat{\cal Z}_k(z)$ is a matrix-valued analytic function. 
As the determinant of ${\hat Z}_k^\eps$ is constant and $\det {\hat Z}_k^\eps 
(0)=1$ we deduce that $\hat{\cal Z}_k(z)$ is a M\"obius transformation.

Our next step is to show that $Z_k^\eps$ also converges for some subsequence of 
$\eps\to 0$. To this end we will show that $Z_k^\eps(0)$ 
is bounded independently of $\eps$ and converges.

Denote by $Z_k^\eps(v)(w)$ the image of $w$ by the M\"obius transformation 
$Z_k^\eps(v)$.
First recall that by Theorem~\ref{theoConvalt} we know that 
$Z_k^\eps(v_z)(0)=f^\eps(v_z)\to f(z)$ for $\eps\to 0$, where $v_z\in 
V_K^\eps$ is a vertex nearest to $z$. Further note that $\hat{\cal 
Z}_k(0)=I$ and $\frac{\partial}{\partial z}|_{z=0} \hat{\cal Z}_1(z)(0) 
=\omega_1\not=0$. Thus, for $\xi>0$ small enough the three points 
$w_0:= \hat{\cal Z}_1(0)(0)$, $w_+:= \hat{\cal Z}_1(\xi)(0)$ and  $w_-:= 
\hat{\cal Z}_1(-\xi)(0)$ are pairwise different. Let $v_+,v_-$ be two 
vertices which are nearest to $\xi$ and -$\xi$ respectively. Then 
$Z_1^\eps(0)$ maps $\hat{Z}_1^\eps(0)(0)$, 
$\hat{Z}_1^\eps(v_+)(0)$, $\hat{Z}_1^\eps(v_-)(0)$ (which are points close to 
$w_0$, $w_+$, $w_-$) to points close to $f(0)$, $f(\xi)$, $f(-\xi)$ 
respectively.
This shows that $\lim\limits_{\eps\to 0} 
Z_1^\eps(0)$ exists and is the M\"obius transformation 
which maps $w_0$, $w_+$, $w_-$ to $f(0)$, $f(\xi)$, $f(-\xi)$. By similar 
arguments, we see that the same is also true for the other transformations 
$Z_k^\eps(0)$. Thus, we obtain $C^\infty$-convergence 
${\cal Z}_k=\lim\limits_{\eps\to 0} {Z}_k^\eps$ along some subsequence $\eps\to 
0$.

\begin{theorem}\label{theoConvS}
\begin{equation}\label{eqS}
 {\mathfrak S}[f]= \frac{2 \left((L_2\omega_2 
+L_3\omega_3){\cal S}_1 + (L_1\omega_1-L_3\omega_3){\cal S}_2 
+(-L_1\omega_1-L_2\omega_2){\cal 
S}_3\right)}{3L_1L_2L_3\omega_1\omega_2\omega_3}
\end{equation}
\end{theorem}
This theorem implies that the convergence of $s_k$ holds along every 
subsequence, which proves Corollary~\ref{corconvsk}.

\begin{proof}
 We have already shown that
\begin{equation*}
 {\cal Z}_1(z)= \begin{pmatrix} a(z) & b(z) \\ c(z) & d(z) \end{pmatrix}
\end{equation*}
is a M\"obius transformation and $f(z)={\cal Z}_1(z)(0)=\frac{b(z)}{d(z)}$. 
From~\eqref{eqdkZ} we deduce the relation
\begin{equation*}
 \partial_1{\cal Z}_1= {\cal Z}_1\begin{pmatrix} 0 & \omega_1 \\  \frac{ {\cal 
S}_{2}}{L_1L_{2}\omega_{2}} +\frac{{\cal S}_{3}}{L_1L_{3}\omega_{3}} & 0 
\end{pmatrix},
\end{equation*}
which implies that $b(z)$ and $d(z)$ satisfy the same differential equation 
\[w''=\frac{\omega_1}{L_1L_2L_3\omega_2\omega_3}\left( L_3\omega_3{\cal S}_2 
+L_2\omega_2){\cal S}_3\right)w.\] 
Therefore $b'd-bd'=\hat{C}=$constant and $(\frac{b}{d})'=\frac{\hat{C}}{d^2}$.
This implies (with $\omega_1=1$) for the Schwarzian of $f$ that
\begin{align}
 {\mathfrak S}[f] &={\mathfrak S}[\frac{b}{d}]=-2\frac{d''}{d} =  
-\frac{2\omega_1}{L_1L_2L_3\omega_2\omega_3}\left( L_3\omega_3{\cal S}_2 
+L_2\omega_2{\cal S}_3\right).\label{eqS23}
\end{align}
Now using $\omega_1=1$, $L_1\omega_1-L_2\omega_2+L_3\omega_3=0$ and the 
identity~\eqref{eqsumSk}, it is easy to check that that~\eqref{eqS} holds.
\end{proof}

As $f^\eps$ is discrete conformal, we deduce from~\eqref{eqskIm} that ${\cal 
S}_k$ ($k=1,2,3$) are purely imaginary. Therefore, it follows 
from~\eqref{eqS23} that
\begin{equation}
 {\cal S}_1=iL_1\text{Re}(\omega_2\omega_3{\mathfrak S}[f]),\quad
{\cal S}_2=-iL_2\text{Re}(\omega_1\omega_3{\mathfrak S}[f]),\quad
{\cal S}_3=iL_3\text{Re}(\omega_1\omega_2{\mathfrak S}[f]).
\end{equation}

Finally we deduce the $C^\infty$-convergence of the discrete conformal 
maps $f^\eps$.
\begin{proof}[Proof of Theorem~\ref{theoConv}]
 Recall that $f^\eps(v)=Z_1^\eps(v)(0)$. We have already shown that 
the M\"obius transformations
\begin{equation*}
 {Z}_1^\eps(v)= \begin{pmatrix} a^\eps(v) & b^\eps(v) \\ c^\eps(v) & d^\eps(v) 
\end{pmatrix}
\end{equation*}
converge for $\eps\to 0$ in $C^\infty$ to the M\"obius transformation ${\cal 
Z}_1(z)$. Then $d^\eps$ converges to $d$ and  $b^\eps$ converges to $b$ and 
$d\not=0$ as $b(z)/d(z)={\cal Z}_1(z)(0)=f(z)$ and the determinant of ${\cal 
Z}_1(z)$ is nonzero. Consequently, Lemma~\ref{lemconv} implies that 
$f^\eps=b^\eps/d^\eps$ converges in $C^\infty$ to $b/d=f$.
\end{proof}

\section{Remarks on generalizations}\label{SecGen}
Discrete analogues of conformal maps already have a long history. The methods 
for a proof of $C^\infty$-convergence considered in this article for discrete 
conformal based on conformally equivalent triangular meshes also works very 
similarly for circle patterns on hexagonal lattices. Here we take the 
circumcircles of all triangles in $T_K^\eps$ and demand that the intersection 
angles of these circumcircles are preserved. In other words, such discrete maps 
preserve the arguments of the cross-ratios $Q_k$. Note that in this case, 
equations~\eqref{eqq1} and~\eqref{eqq2} are still valid. Only~\eqref{eqq3} has 
to be replaced by a similar equation where certain unitary numbers (quotients 
of $\omega_k$'s) appear instead of quotients of absolute values of $Q_k$'s.
Therefore, given convergence 
(in $C^0$ or $C^1$ say) of such circle patterns and some bounds on the discrete 
Schwarzians, the rest of the proof could be applied with only minor adaptions to 
show $C^\infty$-convergence of regular hexagonal circle patterns. Similar ideas 
have already been worked out in~\cite{LD07} for orthogonal circle patterns with 
square grid combinatorics using other M\"obius invariants and in~\cite{Bue08} 
for the 
more general case of isoradial circle patterns studying the radius function. 
Note that similarly to the method of the proof of~\cite{Bue08}, 
$C^\infty$-approximation of the discrete conformal maps considered in this 
article can also be shown based on estimates for the discrete 
Laplacian of the scale factors $u^\eps$.

Generalizing the notion of a conformal map to include both,  
Definition~\ref{defdcm} and hexagonal circle patterns, we may consider maps 
such that a fixed linear 
combination of the real and the imaginary part of the cross-ratios $Q_k$ 
remains constant: 
\[a\log |q_k| +b\arg q_k= a\log |Q_k| +b\arg Q_k\] 
for some fixed constants $a,b\in\R$ and $\arg q_k, \arg Q_k\in (0,2\pi)$.
Again, equations~\eqref{eqq1} and~\eqref{eqq2} are still 
valid and only~\eqref{eqq3} has to be adapted. This generalized notion of 
discrete conformality has not been studied yet. But if there was a suitable 
convergence result (for example uniform convergence of $f^\eps$ and bounds on 
the discrete Schwarzians) the methods of our proof of $C^\infty$-convergence 
could be applied analogously.

\section*{Acknowledgement}
This research was supported by the DFG Collaborative Research Center 
TRR~109, ``Discretization in Geometry and Dynamics''. The author is grateful to 
A.I.~Bobenko for discussions on the topic of discrete conformal maps on 
hexagonal lattices.


\bibliographystyle{amsalpha} 
\bibliography{paperconv}

\end{document}